\documentclass{article}

\usepackage{amsthm} 
\usepackage{amsmath}
\usepackage{hyperref}
\usepackage{graphicx}
\usepackage{fancyhdr}
\usepackage[utf8]{inputenc}
\usepackage[T1]{fontenc}
\usepackage{tikz}

\newtheorem{theorem}{Theorem}[section]  

\newtheorem{corollary}[theorem]{Corollary}

\theoremstyle{definition}
\newtheorem{definition}[theorem]{Definition}

\usepackage{PRIMEarxiv}
\usepackage[utf8]{inputenc} 
\usepackage[T1]{fontenc}    
\usepackage{hyperref}       
\usepackage{url}            
\usepackage{booktabs}       
\usepackage{amsfonts}       
\usepackage{nicefrac}       
\usepackage{microtype}      
\usepackage{lipsum}
\usepackage{fancyhdr}       
\usepackage{graphicx}       
\graphicspath{{media/}}     

\title{Average Steps Until Absorption on Random Walks on Sea Dragon Trees}

\author{
  John Estes* \\
  Belhaven University \\
  \texttt{jestes@belhaven.edu}
  \and
  \textbf{Tyler Jackson} \\
  Belhaven University \\
  \texttt{tylerjjackson@students.belhaven.edu}
  \AND 
  Zachary Chapman \\
  Belhaven University \\
  \texttt{zack@chapfam.net}
  \and
  \textbf{Lillian Ates} \\
  Belhaven University \\
  \texttt{lillianates@students.belhaven.edu}
}

\begin{document}
\maketitle

\begin{abstract}
For a graph $G$ and vertices $u,v$, we define the ASUA of $v$, $t(G,v,u)$, to be the average steps until absorption along a random walk terminating at $u$. We define a sea dragon to be a tree with a unique path $P$ such that if $d(u) \geq 3$ for some vertex $u$, then $u \in V(P)$. We use Markov chains to determine $t(G,v,u)$ for all vertices of several classes of sea dragons, a broad subclass of trees. Additionally, we give several results on equations related to ASUAs on general graphs. 
\end{abstract}


\section{Introduction to Random Walks on Graphs}

A common reinforcement learning application focuses on assisting an AI agent in navigating a grid world to find the optimal path from one location to another \cite{barto, russell}. Consider an autonomous agent in maze, but instead of learning an optimal policy, the agent traverses through the maze challenged to make decisions at each juncture randomly. We will address the following question: On average, how long does it take for the agent to traverse the maze?

\input{tikz/maze}

\par Such a maze is represented as a simple, undirected graph $G = (V,E)$ with a vertex at each grid coordinate as pictured in Figure 1. For vertices $u, v$, $G$ has edge $e =uv$ if the robot can traverse directly from $u$ to $v$. For $v \in V$ and $S \subseteq G$,  we denote $N(v)$ to be the neighborhood of $v$, $d(v) = |N(v)|$ to be the degree, and $G[S]$ to be the subgraph of $G$ induced by $S$.

\par Let $W = \{v,\ldots,u\}$ be a random walk with absorbing vertex $u$ on $G$ starting at $v$, representing the agent's walk through the maze. We define the ``ASUA of $v$" to be the average steps until absorption starting at $v$ denoted as $t(G,v,u)$. If context is implied, we will denote $t(G,v,u)$ as $t(v)$. In this paper, we will calculate $t(G,v,u)$ for several classes of graphs through use of Markov chains.

\subsection{Properties of Markov Chains and Absorbing Points}

Markov chains model stochastic behavior of multiple states over multiple transitions. This model is built with a transition matrix $T = [p_{i,j}]$ where $p_{i,j}$ is the probability that state $i$ can transition to state $j$ in one transition. This model is useful as $T^k$ yields probabilities of states after $k$ transitions. We refer the reader to Grinstead and Snell \cite{grinstead}.

It is commonly known, that in Markov chains with absorbing states (referred to as absorbing Markov chains), the transition matrix, $T$, can be written in canonical form as

\[
T = 
\begin{bmatrix}
  Q & R \\
  0 & I
\end{bmatrix}.
\]

The sub-matrix, $Q$, represents the transitions of transient states, and we seek to find $\lim_{k \rightarrow \infty}Q^k$. It can be shown this limit converges to $N = (I - Q)^{-1}$ \cite{cheteyan}. Let $[1]$ be a column vector all 1's and $[t]$ be a column vector where $t_i$ is the ASUA at state $i$. It is commonly known that $N \cdot [1] = [t]$.

In 1994, Lovasz showed that, for a random walk on G, $T = D \cdot A$, where $D = [d_{i,j}]$ is a diagonal matrix such that $d_{i,i} = 1/d(v_i)$ for vertex $i$ and 0 otherwise, and $A$ is the adjacency matrix of $G$ \cite{lovasz}. . 

As an example, for the graph $G$ shown in Figure \ref{c4u}, we demonstrate matrices $D$, $A$, and $T = D \cdot A$ in Figure \ref{fig:c4umatrix}. The ASUAs of vertices $v_1, v_2, v_3$ and $v_4$ are 13, 14, 10, and 13 respectively, shown in Figure \ref{fig:c4N}.

\begin{figure}[h]
\begin{center}
\includegraphics[width=6in]{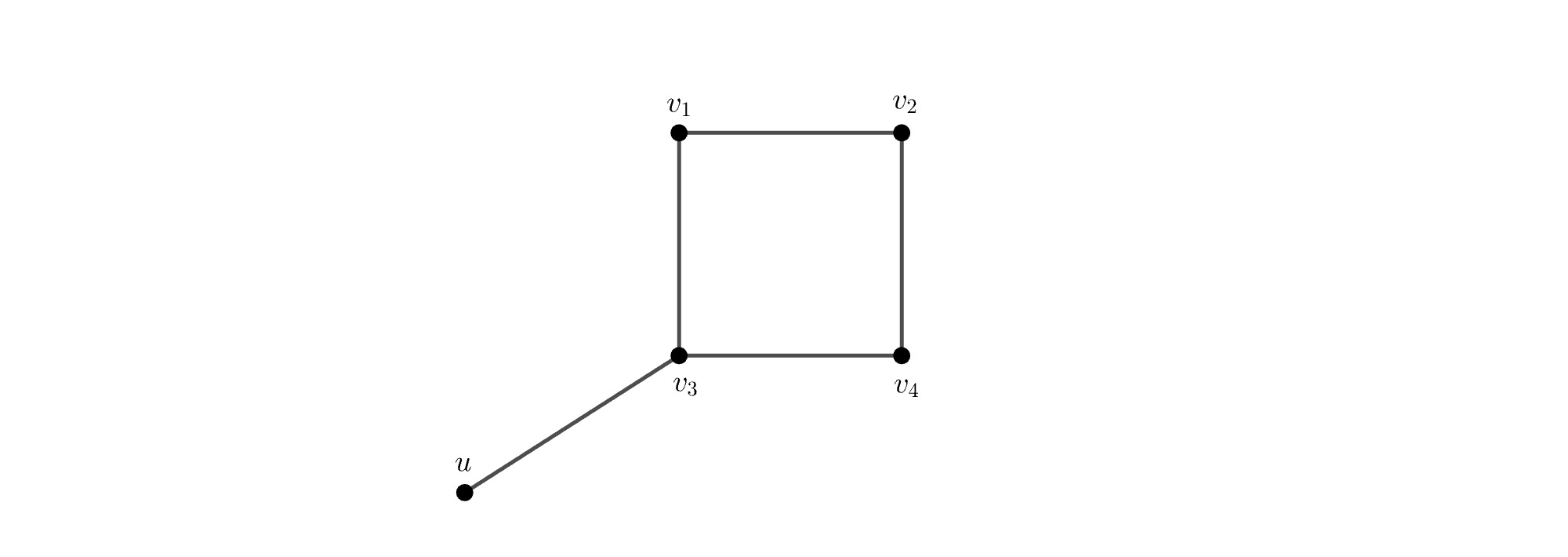} 
\caption{Introductory Example Graph $G$}
\label{c4u}
\end{center}
\end{figure}

\begin{figure}[h]
    \centering
    \[
    D = 
    \begin{bmatrix}
    \frac{1}{2} & 0 & 0 & 0 & 0 \\
    0 & \frac{1}{2} & 0 & 0 & 0 \\
    0 & 0 & \frac{1}{3} & 0 & 0 \\
    0 & 0 & 0 & \frac{1}{2} & 0 \\
    0 & 0 & 0 & 0 & 1 
    \end{bmatrix}
    \quad 
    A = 
    \begin{bmatrix}
    0 & 1 & 1 & 0 & 0 \\
    1 & 0 & 0 & 1 & 0 \\
    1 & 1 & 0 & 0 & 1 \\
    0 & 1 & 1 & 0 & 0 \\
    0 & 0 & 0 & 0 & 1 
    \end{bmatrix}
    \quad 
    T = D \cdot A = 
    \begin{bmatrix}
    0 & \frac{1}{2} & \frac{1}{2} & 0 & 0 \\
    \frac{1}{2} & 0 & 0 & \frac{1}{2} & 0 \\
    \frac{1}{3} & 0 & 0 & \frac{1}{3} & \frac{1}{3} \\
    0 & \frac{1}{2} & \frac{1}{2} & 0 & 0 \\
    0 & 0 & 0 & 0 & 1 
    \end{bmatrix}
    \]
    \caption{The Diagonal Matrix $D$, Adjacency Matrix $A$, and Transition Matrix $T$ of $G$}
    \label{fig:c4umatrix}
\end{figure}

\begin{figure}[h]
    \centering
    \[
    N = (I - Q)^{-1} =
    \begin{bmatrix}
    4 & 4 & 3 & 2\\
    \frac{7}{2} & 5 & 3 & \frac{5}{2} \\
    \frac{5}{2} & 3 & 3 & \frac{3}{2} \\
    3 & 4 & 3 & 3
    \end{bmatrix}
    \quad 
    [t] = 
    \begin{bmatrix}
    13  \\
    14 \\
    10 \\
    13 
    \end{bmatrix}    
    \]
    
    \caption{The Matrix $N = (I-Q)^{-1}$, and ASUA Matrix $[t]$ of $G$}
    \label{fig:c4N}
\end{figure}

\subsection{Our Proof Method}

\noindent For a given graph, calculating $N = (I - Q)^{-1}$ is a standard operation, but calculating $N$ for a class of graphs is not a direct exercise. However, $N^{-1} = I - Q$ is rather straightforward, and since $N$ is an invertible matrix, the equation $N \cdot [1] = [t]$ has a unique solution as does $N^{-1} \cdot [t] = [1]$. 

\noindent For Theorems \ref{thm:sd1} and \ref{thm:sd4}, we will provide a potential solution $[t]$ and show that this potential solution satisfies $N^{-1} \cdot [t] = [1]$ in all cases. Since there is a unique solution to this equation, our proposed solution is the true solution.

\noindent In Section \ref{sec:asuaeqs}, we demonstrate how $N^{-1} \cdot [t] = [1]$ leads to a set of constraint equations with variables $t(v_1), t(v_2), \ldots, t(v_{n-1})$ that a potential solution must satisfy. We refer to these equations as \emph{ASUA equations}.
\section{Preliminary Results of ASUA Equations}
\label{sec:asuaeqs}

\subsection{ASUAs of Neighbors}

\noindent Before presenting our main results, we identify several important results that we will rely on in subsequent sections. We define, for vertex set $S \subseteq V$, $t_{\mu}(S)$ to be the average of ASUAs of vertices of $S$. That is, 
\[t_{\mu}(S) = \dfrac{1}{|S|}\left( \sum_v t(G,v,u) \right),\]

where $u$ is an absorbing vertex for a random walk on $G$. Approaching ASUAs through use of $t_{\mu}(N(v))$ allows us to utilize graph structures to calculate $[t]$ more directly.

\begin{theorem}
\label{thm_mu}
Let $G$ be a graph and $v, u \in V(G)$, and let $t(v) = T(G,v,u)$.  Then \[t(v) = t_{\mu}(N(v)) + 1.\]
\end{theorem}

\begin{proof}
    Let $d(v) = k$ and $N(v) = \{v_1, \ldots, v_k\}$. Then the row corresponding to $v$, $r_v$, of $N^{-1} = I - Q$ has entry 1 at the index corresponding to $v$, $-\frac{1}{k}$ for entries corresponding to $v_i$ for all $i$, and 0 elsewhere. 

    Thus $r_v\cdot t^T = t(v) - \frac{1}{k}(t(v_1) + t(v_2) + \ldots + t(v_k)) = 1.$ Hence $t(v) = t_{\mu}(N(v)) + 1$.
\end{proof}

Theorem \ref{thm_mu}, relates $t(v)$ to ASUAs of other vertices. We refer to this characterization as the ASUA equation of $t(v)$. If $d(v) = 1$, we can quickly determine the ASUA equation $t(v)$.

\begin{corollary}
\label{cor_degreeone}
    Let $N(u) = \{v\}$ and $t(v) = T(G,v,x)$. Then $t(u) = t(v) + 1$.
\end{corollary}

Similarly,
\begin{corollary}
\label{cor_sd1}
    Let $v \in V$ such that $N(v) = \{x,y,u\}$ where $d(u) = 1$. Then $t(v) = \frac{1}{2}(t(x) + t(y)) + 2$.
\end{corollary}

\begin{proof}
    From Theorem \ref{thm_mu} and Corollary \ref{cor_degreeone}, 
    \[t(v) = \frac{1}{3}(t(x) + t(y) + t(u)) + 1, \text{\hspace{.5in}} t(u) = t(v) + 1.\]

Thus, $3t(v) = t(x) + t(y) + t(u) + 3 = t(x) + t(y) + (t(v) + 1) + 3$. Solving for $t(v)$, we obtain $t(v) = \frac{1}{2}(t(x) + t(y)) + 2$.
    
\end{proof}

\subsection{ASUAs of Stems}

\noindent Suppose $U = \{u_1, u_2, \ldots, u_l\} \subseteq V(G)$ such that $G[U]$ is a path, $d(u_1) = 1$, $d(u_i) = 2$ for $2 \leq i \leq l$, and $N(u_l) = \{u_{l-1},v \}$. We say that $U$ is a \emph{stem attached to $v$}. Then we can state one more corollary from Theorem \ref{thm_mu}.

\begin{theorem}
\label{thm_stems}
\noindent Let $G$ be a graph such that $\{u_1, u_2, \ldots, u_l\}$ is a stem attached to $v$. Also let $t(u) = T(G,u,x)$ and for convention $v = u_{l+1}$. Then $t(u_j) = t(v) + l^2 - (j-1)^2$ for $1 \leq j \leq l + 1$.
\end{theorem}

\begin{proof}
By Theorem \ref{thm_mu}, $t(u_1) = t(u_2) + 1$ and $t(u_j)  = \frac{1}{2}(t(u_{j-1} + t(u_{j+1})) + 1$ for $2 \leq j \leq l$.  

\noindent We will prove, by induction, that $t(u_1) = t(u_j) + (j-1)^2$ with $t(u_1) = t(u_1) + (1-1)^2$ and $t(u_1) = t(u_2) + (2-1)^2$ confirming the base cases.

\noindent Suppose that for some $a$, $1 \leq a \leq l$, $t(u_1) = t(u_a) + (a-1)^2$. Then,

\begin{align*}
    t(u_1) & = t(u_a) + (a-1)^2 \\
        & =  \frac{1}{2}(t(u_{a-1}) + t(u_{a+1})) + 1 + (a-1)^2 \\
    2t(u_1) & = t(u_{a-1}) + t(u_{a+1}) + 2 + 2a^2 -4a + 2 \\
            & = (t(u_1) - (a-2)^2) + t(u_{a+1}) + 2a^2 - 4a + 4 \\
    t(u_1) & = t(u_{a+1}) + a^2.
\end{align*}

\noindent Thus $t(u_1) = t(u_j) + (j-1)^2$ for $1 \leq j \leq l$. To complete the proof, note that $t(v) = t(u_{l+1}) = t(u_1) - l^2$, and so $t(u_j) =  t(u_1) - (j-1)^2 = (t(v) + l^2) - (j-1)^2$.

\end{proof}

\begin{corollary}
\label{cor_sd4}
    Let $N(v) = \{x,y\}\cup\{u_1,\ldots,u_r\}$, and for $1\leq i \leq r$, let $S_i$ be a stem of length $c_i$ attached to $v$ via $u_i$. Let $c_1 + c_2 + \ldots c_r = d$. Then 
    \[t(v) = \dfrac{1}{2}(t(x) + t(y)) + (d+1).\]
\end{corollary}

\begin{proof}
    From Theorem \ref{thm_mu}, \ref{thm_stems}, and Corollary \ref{cor_degreeone}, 
    \[t(v) = \frac{1}{r+2}(t(x) + t(y) + \sum_{i=1}^r t(u_i)) + 1, \text{\hspace{.5in}} t(u_i) = t(v) + 2c_i  - 1.\]

Thus we have 
\begin{align*}
    (r+2)t(v) = & t(x) + t(y) + \sum_{i=1}^r t(u_i) + r + 2 \\
                = & t(x) + t(y) + rt(v) + 2d - r + r + 2 \\
    2t(v) = & t(x) + t(y) + 2d + 2 \\
    t(v) = & \frac{1}{2}(t(x) + t(y)) + (d+1).
\end{align*}
    
\end{proof}

\section{Paths and ASUAs of Cycles}

Theorem \ref{thm_stems} actually helps us determine ASUAs for vertices on the path graph on $n$ vertices, $P_n$.

\begin{corollary}
    Let $P_n$ be a path graph with vertices $\{v_1,\ldots,v_n\}$. Then $t(P_n,v_i,v_n) = (n-1)^2 - (i-1)^2$ \text{for} $1 \leq i \leq n - 1$.
\end{corollary}

\begin{proof}
    The path $P_n$ can be constructed as a stem $\{v_1,\ldots,v_{n-1}\}$ attached to $v_n$. By Theorem \ref{thm_stems}, $t(v_i) = t(v_n) + (n-1)^2 - (i-1)^2.$ Since $t(v_n) = 0$, the result is confirmed.
\end{proof}

\noindent As an example of the usefulness of ASUA equations, we present Theorem \ref{thm_cycles}. 

\begin{theorem}
\label{thm_cycles}
    Let $C_n$ be a cycle graph with vertices $\{v_1,v_2,\ldots,v_n\}$. Then $t(v_i)=t(C_n,v_i,v_n) = i(n-i)$ for $1 \leq i\leq n-1$.
\end{theorem} 

\begin{proof}
\noindent By Theorem \ref{thm_mu}, $t(v_i) = \dfrac{1}{2}(t(v_{i-1}) + t(v_{i+1})) + 1$ for $1 \leq i \leq n-1$ with arithmetic on the indices being modulo $n$. Since $t(v_n) = 0$,

\begin{equation}\label{eq_cyclev1} t(v_1) = \dfrac{1}{2}(t(v_{n}) + t(v_{2})) + 1 = \dfrac{1}{2}t(v_2) + 1. \end{equation} 

\par We will prove by induction that

\begin{equation}\label{eq_cycleind} t(v_i) = i(t(v_1)-(i-1)) \text{ for } i \geq 1. \end{equation}

If $i = 1$, this equation resolves to $t(v_1) = t(v_1)$. By (\ref{eq_cyclev1}), $t(v_2) = 2t(v_1) - 2 = 2(t(v_1)-1)$, and so the statement holds for $i = 2$. Suppose that $t(v_{j'}) = j'(t(v_1) + (j'-1))$ for $2 \leq j' < j$. Then, $t(v_{j-1})  =  \dfrac{1}{2}(t(v_{j}) + t(v_{j-2})) + 1$ and so 
    \begin{align*}
      2t(v_{j-1}) & =  t(v_j) + t(v_{j-2}) + 2 \\
     t(v_j) & = 2t(v_{j-1}) - t(v_{j-2}) - 2, \\  
\end{align*}

and by induction

\begin{align*}
    t(v_j) & = 2t(v_{j-1}) - t(v_{j-2}) - 2 \\
     & = 2(j-1)(t(v_1) - (j-2)) - (j-2)(t(v_1) - (j-3)) - 2 \\
     & = jt(v_1) - (j-2)(j+1) - 2 \\
     & = j(t(v_1) - (j-1) )
\end{align*}

\noindent as was desired to be shown.

\par Thus $t(v_{n-1}) = (n-1)(t(v_{n-1} - (n-2))$. Note that by symmetry of $C_n$, $t(v_{n-i}) = t(v_i)$ for $1 \leq i \leq n-1$. Thus $t(v_{n-1}) = t(v_1)$. Hence $t(v_1) = (n-1)(t(v_1)-(n-2)) \Rightarrow t(v_1) = n-1$. Thus (\ref{eq_cycleind}) becomes 

\[t(v_i)=t(C_n,v_i,v_n) = i(n-i) \text{ for } 1 \leq i \leq n-1.\]
    
\end{proof}

\section{Trees and Sea Dragons}

\noindent We now turn our attention to a class of trees we refer to as sea dragons. We define $G$ to be a \emph{sea dragon} if $G$ is a tree with a unique path $P$ (called the \emph{spine}) such that any vertex $v$ with $d(v) \geq 3$ lies on $P$. Throughout the paper, we enumerate the vertices of the spine as $\{v_1,\ldots, v_n\}$. There are several classes of sea dragons that we will investigate, referred to as SD1, SD2, and SD3 pictured in Figures \ref{SSD1}, \ref{SSD2}, and \ref{SSD3}.

\begin{definition}
A sea dragon $G$ is in class SD1 if $V(G) = \{v_1,\ldots,v_n\} \cup \{u_{k_1}, u_{k_2}, \ldots, u_{k_a}\}$,
where each $u_{k_i}$ is a leaf adjacent to $v_{k_i}$ for $1 \leq i \leq a$. 
We denote $G$ as $T(n,\{k_1,k_2,\ldots,k_a\})$.
\end{definition}

\begin{figure}[h] 
\begin{center}
\begin{tikzpicture}[scale=1]
\coordinate (A) at (-3,0);
\coordinate (B) at (-2,0);
\coordinate (C) at (-1,0);
\coordinate (D) at (0,0);
\coordinate (E) at (1,0);
\coordinate (F) at (2,0);
\coordinate (G) at (3,0);
\coordinate (H) at (-2,1);
\coordinate (I) at (0,1);
\coordinate (J) at (1,1);

\draw [fill=black,thick] (A) circle [radius=1pt];
\draw [fill=black,thick] (B) circle [radius=1pt];
\draw [fill=black,thick] (C) circle [radius=1pt];
\draw [fill=black,thick] (D) circle [radius=1pt];
\draw [fill=black,thick] (E) circle [radius=1pt];
\draw [fill=black,thick] (F) circle [radius=1pt];
\draw [fill=black,thick] (G) circle [radius=1pt];
\draw [fill=black,thick] (H) circle [radius=1pt];
\draw [fill=black,thick] (I) circle [radius=1pt];
\draw [fill=black,thick] (J) circle [radius=1pt];

\draw [line width=1pt] (A) to (B);
\draw [line width=1pt] (B) to (C);
\draw [line width=1pt, dash pattern=on 1pt off 2pt] (C) to (D);
\draw [line width=1pt] (D) to (E);
\draw [line width=1pt, dash pattern=on 1pt off 2pt] (E) to (F);
\draw [line width=1pt] (F) to (G);
\draw [line width=1pt] (B) to (H);
\draw [line width=1pt] (D) to (I);
\draw [line width=1pt] (E) to (J);

\node at (-3,-.3) {$v_1$};
\node at (-2,-.3) {$v_2$};
\node at (-1,-.3) {$v_3$};
\node at (0,-.3) {$v_{k_1}$};
\node at (1,-.3) {$v_{k_2}$};
\node at (3,-.3) {$v_n$};
\node at (-2,1.3) {$u_2$};
\node at (0,1.3) {$u_{k_1}$};
\node at (1,1.3) {$u_{k_2}$};

\end{tikzpicture}
\end{center}
\caption{Category SD1: $T(n,\{2,k_1,k_2\})$}
\label{SSD1}
\end{figure}
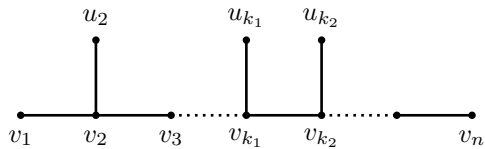

\begin{definition}
A sea dragon $G$ is in class SD2 if $V(G) = \{v_1,\ldots,v_n\} \cup \{u_{k_1}, u_{k_2}, \ldots, u_{k_b}\}$,
where each $u_i$ is a leaf adjacent to the same spine vertex $v_k$ for $1 \leq i \leq b$. 
We denote $G$ as $T(n,(k,b))$.
\end{definition}

\begin{figure}[h] 
\begin{center}
\begin{tikzpicture}[scale=1]
\coordinate (A) at (-3,0);
\coordinate (B) at (-2,0);
\coordinate (C) at (-1,0);
\coordinate (D) at (0,0);
\coordinate (E) at (1,0);
\coordinate (F) at (2,0);
\coordinate (G) at (3,0);
\coordinate (H) at (-.5,1);
\coordinate (I) at (0,1);
\coordinate (J) at (.5,1);

\draw [fill=black,thick] (A) circle [radius=1pt];
\draw [fill=black,thick] (B) circle [radius=1pt];
\draw [fill=black,thick] (C) circle [radius=1pt];
\draw [fill=black,thick] (D) circle [radius=1pt];
\draw [fill=black,thick] (E) circle [radius=1pt];
\draw [fill=black,thick] (F) circle [radius=1pt];
\draw [fill=black,thick] (G) circle [radius=1pt];
\draw [fill=black,thick] (H) circle [radius=1pt];
\draw [fill=black,thick] (I) circle [radius=1pt];
\draw [fill=black,thick] (J) circle [radius=1pt];

\draw [line width=1pt] (A) to (B);
\draw [line width=1pt, dash pattern=on 1pt off 2pt] (B) to (C);
\draw [line width=1pt] (C) to (D);
\draw [line width=1pt] (D) to (E);
\draw [line width=1pt, dash pattern=on 1pt off 2pt] (E) to (F);
\draw [line width=1pt] (F) to (G);
\draw [line width=1pt] (D) to (H);
\draw [line width=1pt] (D) to (I);
\draw [line width=1pt] (D) to (J);

\node at (-3,-.3) {$v_1$};
\node at (-2,-.3) {$v_2$};
\node at (-1,-.3) {$v_{k-1}$};
\node at (0,-.3) {$v_{k}$};
\node at (1,-.3) {$v_{k+1}$};
\node at (3,-.3) {$v_n$};
\node at (-.5,1.3) {$u_1$};
\node at (0,1.3) {$u_2$};
\node at (.5,1.3) {$u_{b}$};

\end{tikzpicture}
\end{center}
\caption{Category SD2: $T(n,(k,b))$}
\label{SSD2}
\end{figure}
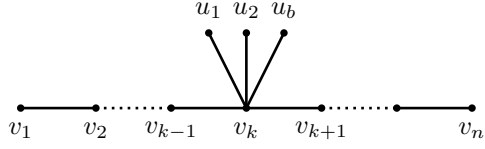

\begin{definition}
A sea dragon $G$ is in class SD3 if $V(G) = \{v_1,\ldots,v_n\} \cup \{u_{k_1}, u_{k_2}, \ldots, u_{k_a}\}$, 
where $\{u_1,\ldots,u_c\}$ forms a stem attached to $v_k$, with $d(u_1) = 1$. 
We denote $G$ as $T(n, k^{(c)})$.
\end{definition}

\begin{figure}[h] 
\begin{center}
\begin{tikzpicture}[scale=1]
\coordinate (A) at (-3,0);
\coordinate (B) at (-2,0);
\coordinate (C) at (-1,0);
\coordinate (D) at (0,0);
\coordinate (E) at (1,0);
\coordinate (F) at (2,0);
\coordinate (G) at (3,0);
\coordinate (H) at (0,.5);
\coordinate (I) at (0,1);
\coordinate (J) at (0,1.5);

\draw [fill=black,thick] (A) circle [radius=1pt];
\draw [fill=black,thick] (B) circle [radius=1pt];
\draw [fill=black,thick] (C) circle [radius=1pt];
\draw [fill=black,thick] (D) circle [radius=1pt];
\draw [fill=black,thick] (E) circle [radius=1pt];
\draw [fill=black,thick] (F) circle [radius=1pt];
\draw [fill=black,thick] (G) circle [radius=1pt];
\draw [fill=black,thick] (H) circle [radius=1pt];
\draw [fill=black,thick] (I) circle [radius=1pt];
\draw [fill=black,thick] (J) circle [radius=1pt];

\draw [line width=1pt] (A) to (B);
\draw [line width=1pt, dash pattern=on 1pt off 2pt] (B) to (C);
\draw [line width=1pt] (C) to (D);
\draw [line width=1pt] (D) to (E);
\draw [line width=1pt, dash pattern=on 1pt off 2pt] (E) to (F);
\draw [line width=1pt] (F) to (G);
\draw [line width=1pt] (D) to (H);
\draw [line width=1pt, dash pattern=on 1pt off 2pt] (H) to (I);
\draw [line width=1pt] (I) to (J);

\node at (-3,-.3) {$v_1$};
\node at (-2,-.3) {$v_2$};
\node at (-1,-.3) {$v_{k-1}$};
\node at (0,-.3) {$v_{k}$};
\node at (1,-.3) {$v_{k+1}$};
\node at (3,-.3) {$v_n$};
\node at (-.5,.5) {$u_c$};
\node at (-.5,1) {$u_2$};
\node at (-.5,1.5) {$u_1$};

\end{tikzpicture}
\end{center}
\caption{Category SD3: $T(n,k^{(c)})$}
\label{SSD3}
\end{figure}
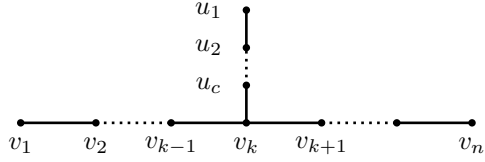

\noindent We can now present ASUA equations for SD1, SD2, and SD3.

\begin{theorem}[ASUAs of SD1 Trees]
\label{thm:sd1}
\noindent Let $G = T(n,\{k_1,\ldots, k_a\})$, $t(v_i) = t(G,v_i,v_n)$, and, by convention, $k_0 = 1$. Then for $0 \leq s \leq a-1$,

\[
t(v_i) = 
\begin{cases}
    n^2- i^2 +2(a-1)n -2(s-1)i -2\displaystyle\sum_{j=s+1}^a k_j, & \text{for } k_s \leq i \leq k_{s+1} \\
    n^2- i^2 + 2(a-1)(n-i), & \text{for } k_{a} \leq i \leq n-1 
\end{cases}
\]
\end{theorem}

\begin{proof}

\noindent Let $g(s,i) = n^2- i^2 +2(a-1)n -2(s-1)i -2\sum_{j=s+1}^a k_j, \text{ for } k_s \leq i \leq k_{s+1}$ and $g(a,i)=  n^2- i^2 + 2(a-1)(n-i), \text{ for } k_{a} \leq i \leq n-1$. Note that if we allow $\sum_{j=a+1}^a k_j$ to take the value of 0, then $g(s,i) = n^2- i^2 +2(a-1)n -2(s-1)i -2\sum_{j=s+1}^a k_j$ for all values of $i$ and $s$. To help communicate structure needed in our argument, we say vertices \{$v_{k_s}, v_{k_s + 1}, \ldots, v_{k_{s+1} -1}$\} form a \emph{section} of the graph.

\noindent We will show that $g(s,i)$ is consistent solution to ASUA equations associated with the graph in all cases. We first look at the end cases of $i=1$ and $i = n-1$.

\noindent Let $i=1$ and $k_1 > 2$. Then $t(v_1)  = t(v_2) + 1 = g(0,2) + 1 = n^2 - 1^2 + 2(a-1)n + 2 - 2\sum_{j = 1}^a k_j = g(0,1)$. Likewise if $k_1 = 2$, then $t(v_2) + 1 = g(1,2) + 1 = n^2 - 1^2 - 3 + 2(a-1)n - 0 - 2\sum_{j = 1}^a k_j + 1 = n^2 - 1^2 - 2(a-1)n - 2(0-1) - 2\sum_{j = 1}^a k_j = g(0,1)$.

\noindent Now let $i = n-1$ and $k_a < n-1$. Then $t(v_{n-1}) = \dfrac{1}{2}t(v_{n-2}) + 1 = \dfrac{1}{2}g(a, n-2) + 1$. Thus

\begin{align*}
    t(v_{n-1}) & = \dfrac{1}{2}(n^2 - (n-2)^2 + 2(a-1)n - 2(a-1)(n-2)) + 1 \\
    & = \dfrac{1}{2}(4n - 4 + 2(a-1)(n - (n-2))) + 1 \\
    & = 2n -2 + 2(a-1) + 1 \\
    & = n^2 - n^2 + 2n -1 + 2(a-1)(n-(n-1)) = g(a,n-1).
\end{align*}

If $k_a = n-1$. Thus $t(v_{n-1}) = \dfrac{1}{2}t(v_{n-2}) + 2 = \dfrac{1}{2}g(a-1, n-2) + 2$, and so

\begin{align*}
    t(v_{n-1}) & = \dfrac{1}{2}(n^2 - (n-2)^2 + 2(a-1)n - 2(a-2)(n-2) - 2k_a) + 2 \\
    & = 2n -2 + (a-1)n - (a-2)(n-2) - (n-1) + 2 \\
    & = n^2 - (n-1)^2 + 2a - 4 + 2 \\
    & = n^2 - (n-1)^2 + 2(a-1) = g(a,n-1).
\end{align*}

\noindent We now turn to $t(v_i)$ for $2 \leq i \leq n-2$. The local structure of $v_{i-1}, v_i,$ and $v_{i+1}$ falls into one of four types pictured in Figure \ref{tikz_sd1proof}.

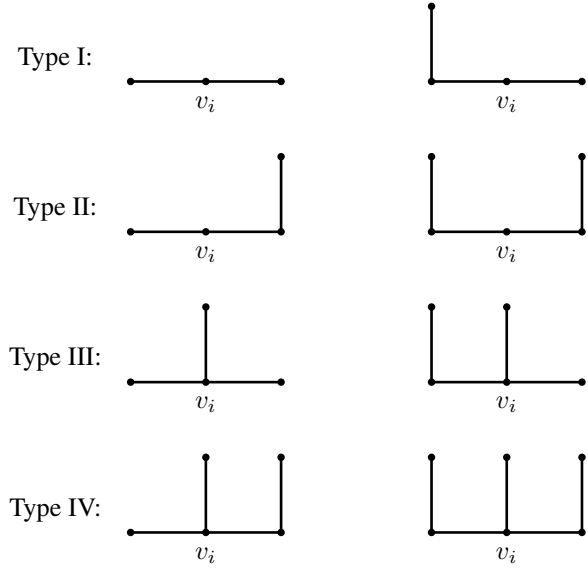
\begin{figure}[h] 
\label{tikz_sd1proof}
\begin{center}
\begin{tikzpicture}[scale=1]
\coordinate (A) at (-3,4);
\coordinate (B) at (-2,4);
\coordinate (C) at (-1,4);
\coordinate (D) at (0,4);
\coordinate (E) at (1,4);
\coordinate (F) at (2,4);
\coordinate (G) at (3,4);

\coordinate (I) at (1,5);
\coordinate (J) at (3,5);

\draw [fill=black,thick] (A) circle [radius=1pt];
\draw [fill=black,thick] (B) circle [radius=1pt];
\draw [fill=black,thick] (C) circle [radius=1pt];
\draw [fill=black,thick] (E) circle [radius=1pt];
\draw [fill=black,thick] (F) circle [radius=1pt];
\draw [fill=black,thick] (G) circle [radius=1pt];
\draw [fill=black,thick] (I) circle [radius=1pt];

\draw [line width=1pt] (A) to (B);
\draw [line width=1pt] (B) to (C);

\draw [line width=1pt] (E) to (F);
\draw [line width=1pt] (F) to (G);
\draw [line width=1pt] (E) to (I);

\node at (-4,4.3) {Type I:};
\node at (-2,3.7) {$v_i$};
\node at (2,3.7) {$v_i$};


\coordinate (A1) at (-3,2);
\coordinate (B1) at (-2,2);
\coordinate (C1) at (-1,2);
\coordinate (D1) at (0,2);
\coordinate (E1) at (1,2);
\coordinate (F1) at (2,2);
\coordinate (G1) at (3,2);

\coordinate (L1) at (-1,3);
\coordinate (M1) at (1,3);
\coordinate (N1) at (3,3);

\draw [fill=black,thick] (A1) circle [radius=1pt];
\draw [fill=black,thick] (B1) circle [radius=1pt];
\draw [fill=black,thick] (C1) circle [radius=1pt];
\draw [fill=black,thick] (E1) circle [radius=1pt];
\draw [fill=black,thick] (F1) circle [radius=1pt];
\draw [fill=black,thick] (G1) circle [radius=1pt];

\draw [fill=black,thick] (L1) circle [radius=1pt];
\draw [fill=black,thick] (M1) circle [radius=1pt];
\draw [fill=black,thick] (N1) circle [radius=1pt];

\draw [line width=1pt] (A1) to (B1);
\draw [line width=1pt] (B1) to (C1);
\draw [line width=1pt] (E1) to (F1);
\draw [line width=1pt] (F1) to (G1);

\draw [line width=1pt] (C1) to (L1);
\draw [line width=1pt] (E1) to (M1);
\draw [line width=1pt] (G1) to (N1);

\node at (-4,2.3) {Type II:};
\node at (-2,1.7) {$v_i$};
\node at (2,1.7) {$v_i$};


\coordinate (A2) at (-3,0);
\coordinate (B2) at (-2,0);
\coordinate (C2) at (-1,0);
\coordinate (D2) at (0,0);
\coordinate (E2) at (1,0);
\coordinate (F2) at (2,0);
\coordinate (G2) at (3,0);
\coordinate (H2) at (4,0);

 \coordinate (I2) at (-2,1); 
\coordinate (L2) at (1,1);
\coordinate (J2) at (2,1);

\draw [fill=black,thick] (A2) circle [radius=1pt];
\draw [fill=black,thick] (B2) circle [radius=1pt];
\draw [fill=black,thick] (C2) circle [radius=1pt];
\draw [fill=black,thick] (E2) circle [radius=1pt];
\draw [fill=black,thick] (F2) circle [radius=1pt];
\draw [fill=black,thick] (G2) circle [radius=1pt];

\draw [fill=black,thick] (L2) circle [radius=1pt];
\draw [fill=black,thick] (I2) circle [radius=1pt];
\draw [fill=black,thick] (J2) circle [radius=1pt];

\draw [line width=1pt] (A2) to (B2);
\draw [line width=1pt] (B2) to (C2);
\draw [line width=1pt] (E2) to (F2);
\draw [line width=1pt] (F2) to (G2);
\draw [line width=1pt] (L2) to (E2);
\draw [line width=1pt] (I2) to (B2);
\draw [line width=1pt] (J2) to (F2);

\node at (-4,.3) {Type III:};
\node at (-2,-.3) {$v_i$};
\node at (2,-.3) {$v_i$};


\coordinate (A3) at (-3,-2);
\coordinate (B3) at (-2,-2);
\coordinate (C3) at (-1,-2);
\coordinate (D3) at (-2,-1);
\coordinate (E3) at (-1,-1);
\coordinate (F3) at (-1,-2);
\coordinate (H3) at (0,-2);
\coordinate (I3) at (1,-2);
\coordinate (J3) at (2,-2);
\coordinate (K3) at (3,-2);
\coordinate (L3) at (1,-1);
\coordinate (M3) at (2,-1);
\coordinate (N3) at (3,-1);

\draw [fill=black,thick] (A3) circle [radius=1pt];
\draw [fill=black,thick] (B3) circle [radius=1pt];
\draw [fill=black,thick] (C3) circle [radius=1pt];
\draw [fill=black,thick] (D3) circle [radius=1pt];
\draw [fill=black,thick] (E3) circle [radius=1pt];
\draw [fill=black,thick] (I3) circle [radius=1pt];
\draw [fill=black,thick] (J3) circle [radius=1pt];
\draw [fill=black,thick] (K3) circle [radius=1pt];
\draw [fill=black,thick] (L3) circle [radius=1pt];
\draw [fill=black,thick] (M3) circle [radius=1pt];
\draw [fill=black,thick] (N3) circle [radius=1pt];

\draw [line width=1pt] (A3) to (B3);
\draw [line width=1pt] (B3) to (C3);
\draw [line width=1pt] (D3) to (B3);
\draw [line width=1pt] (E3) to (C3);
\draw [line width=1pt] (I3) to (J3);
\draw [line width=1pt] (J3) to (K3);
\draw [line width=1pt] (I3) to (L3);
\draw [line width=1pt] (J3) to (M3);
\draw [line width=1pt] (K3) to (N3);

\node at (-4,-1.7) {Type IV:};
\node at (-2,-2.3) {$v_i$};
\node at (2,-2.3) {$v_i$};

\end{tikzpicture}
\end{center}
\caption{Structure types for SD1}
\label{grotlem2}
\end{figure}

\noindent The Type I structure is characterized by $v_{i-1}, v_i,$ and $v_{i+1}$ all belonging to the same section. In this case, $t(v_i) = \dfrac{1}{2}(g(s,i-1) + g(s,i+1)) + 1$. Thus

\begin{align*}
    t(v_i) & = \dfrac{1}{2}(g(s,i-1) + g(s,i+1)) + 1 \\
          & = n^2 + 2(a-1)n - 2\displaystyle \sum_{j = s+1}^a k_j - (i^2 + 1) \\
            & \hspace{.5in} + \dfrac{1}{2}(-2(s-1)(i - 1) - 2(s-1)(i+1)) + 1 \\
            & = n^2 - i^2 - 2(a-1)n - 2(s-1)i - 2\displaystyle \sum_{j = s+1}^a k_j - 1 + 1\\
            & = g(s,i).
\end{align*}

\noindent Type II structure is characterized by $v_{i-1}$ and $v_i$ belonging to the $s^{th}$ section while $v_{i+1}$ in the $(s+1)^{th}$ section. In this case, $v_i = k_{s+1} - 1$ and $t(v_i) = \dfrac{1}{2}(g(s,i-1) + g(s,i+1)) + 1$. Thus,

\begin{align*}
    t(v_i) & = \dfrac{1}{2}(g(s,i-1) + g(s+1,i+1)) + 1 \\
          & = n^2 + 2(a-1)n - i^2 -1 \, +  \\
            & \hspace{.5in} + \dfrac{1}{2}(-2(s-1)(i - 1) - 2s(i+1) - 2\displaystyle \sum_{j = s+1}^a k_j - 2\displaystyle \sum_{j = s+2}^a k_j) + 1 \\
            & = n^2 - i^2 - 2(a-1)n - 2si + 2i - i - 1- 2\displaystyle \sum_{j = s+2}^a k_j - k_{s+1} \\
            & = n^2 - i^2 - 2(a-1)n - 2(s-1)i - k_{s+1} + 1 -1 -  2\displaystyle \sum_{j = s+2}^a k_j - k_{s+1} \\
            & = g(s,i).
\end{align*}

\noindent Type III structure is characterized by $v_{i-1}$ belonging to the $(s+1)^{th}$ section while $v_i$ and $v_{i+1}$ belong to the $s^{th}$ section. In this case, $v_i = k_{s}$ and $t(v_i) = \dfrac{1}{2}(g(s-1,i-1) + g(s,i+1)) + 2$. Thus,

\begin{align*}
    t(v_i) & = \dfrac{1}{2}(g(s-1,i-1) + g(s,i+1)) + 2 \\
          & = n^2 + 2(a-1)n - i^2 -1 \, +  \\
            & \hspace{.5in} + \dfrac{1}{2}(-2(s-2)(i - 1) - 2(s-1)(i+1) - 2\displaystyle \sum_{j = s}^a k_j - 2\displaystyle \sum_{j = s+1}^a k_j) + 2 \\
            & = n^2 - i^2 - 2(a-1)n - 1 - 2si + 2i + i - 1- 2\displaystyle \sum_{j = s+1}^a k_j - k_{s} + 2 \\
            & = g(s,i) -2 + k_s - k_s + 2 = g(s,i).
\end{align*}

\noindent Lastly, Type IV structure is characterized by each of $v_{i-1}, v_i$, and $v_{i+1}$ all belonging to different sections. In this case, $i = k_s = k_{s+1} -1$, and $t(v_i) = \dfrac{1}{2}(g(s-1,i-1) + g(s+1,i+1)) + 2$. Thus,

\begin{align*}
    t(v_i) & = \dfrac{1}{2}(g(s-1,i-1) + g(s+1,i+1)) + 2 \\
          & = n^2 + 2(a-1)n - i^2 -1 \, +  \\
            & \hspace{.5in} + \dfrac{1}{2}(-2(s-2)(i - 1) - 2s(i+1) - 2\displaystyle \sum_{j = s}^a k_j - 2\displaystyle \sum_{j = s+s}^a k_j) + 2 \\
            & = n^2 - i^2 - 2(a-1)n - 1 - 2si + 2i + 2 - 2\displaystyle \sum_{j = s+1}^a k_j - k_{s} - (k_{s} + 1) + 2 \\
            & = g(s,i) - 1 + k_s - k_s -1 + 2 = g(s,i).
\end{align*}

\noindent Thus $t(v_i) = g(s,i)$ for all values of $i$ proving the theorem.

\end{proof}

\section{ASUAs of SD2 and SD3}

\noindent Now we present similar theorems for SD2 and SD3 graphs.

\begin{theorem}[ASUAs of SD2 Trees]
\label{thm_sd2}
\noindent Let $t(v_i) = t(T(n,(k,b)),v_i,v_n)$. Then 

\[
t(v_i) = 
\begin{cases}
    n^2- k^2 + 2(b-1)(n-k) + (k+1)^2 - (i-1)^2, & \text{for } 1 \leq i \leq k-1 \\
    n^2- i^2 + 2(b-1)(n-i), & \text{for } k \leq i \leq n-1 
\end{cases}
\]
\end{theorem}

\begin{theorem}[AUSAs of SD3 Trees]
\label{thm_sd3}
\noindent Let $t(v_i) = t(T(n,k^{(c)}),v_i,v_n)$. Then 

\[
t(v_i) = 
\begin{cases}
    n^2- k^2 + 2(c-1)(n-k) + (k+1)^2 - (i-1)^2, & \text{for } 1 \leq i \leq k-1 \\
    n^2- i^2 + 2(c-1)(n-i), & \text{for } k \leq i \leq n-1 
\end{cases}
\]
\end{theorem}

\noindent Surprisingly, ASUA equations for SD2 and SD3 graphs are the same for the same values of $b$ and $c$. To prove Theorems \ref{thm_sd2} and \ref{thm_sd3}, we introduce a more general class of sea dragons, SD4.

\noindent We define the category SD4, $T(n,k,(c_1,c_2,\ldots,c_r))$, of sea dragons as follows. For $1\leq i \leq r$, let $S_i$ be a stem of length $c_i$ attached to $v_k$ via $u_i$. If $G$ is a SD4 sea dragon and $r=d$, then $G$ is also SD2, and if $r = 1$ and $c_1 = d$, then $G$ is class SD3. We now find ASUAs of SD4 sea dragons.

\begin{theorem}[ASUAs of SD4 Trees]
\label{thm:sd4}
\noindent Let $t(v_i) = t(T(n,k,(c_1, c_2, \ldots, c_r)),v_i,v_n)$, and $c_1 + c_2 + \ldots + c_r = d$. Then 

\[
t(v_i) = 
\begin{cases}
    n^2- k^2 + 2(d-1)(n-k) + (k-1)^2 - (i-1)^2, & \text{for } 1 \leq i \leq k-1 \\
    n^2- i^2 + 2(d-1)(n-i), & \text{for } k \leq i \leq n-1 
\end{cases}
\]
\end{theorem}

\begin{proof}
    Let 
    \[
f(i) = 
\begin{cases}
    n^2- k^2 + 2(d-1)(n-k) + (k-1)^2 - (i-1)^2, & \text{for } 1 \leq i \leq k-1 \\
    n^2- i^2 + 2(d-1)(n-i), & \text{for } k \leq i \leq n-1. 
\end{cases}
\]

\noindent As before, we will show that $t(v_i) = f(i)$ for all $i$, $1 \leq i \leq n-1$ through use of ASUA equations providing the unique solution to $N^{-1}\cdot[t]=[1]$, and we have several cases: $i = 1$, $i = n-1$, $1 \leq i \leq k-2$, $k+1 \leq i \leq n-1$, $i = k-1$, and $i = k$. In particular, from Theorem \ref{thm_mu} and Corollary \ref{cor_sd4}, we will rely on the following ASUA equations:  

\begin{align*}
    t(v_1) & = t(v_2) + 1 \\
    t(v_{n-1}) & = \dfrac{1}{2}t(v_{n-1}) \\
    t(v_i) & =  \dfrac{1}{2}(t(v_{i-1}) + t(v_{i+1}))  + 1 \text{ for } 2 \leq i \leq n-2 \text{ and } n \neq k \\
    t(v_k) & = \dfrac{1}{2}(t(v_{k-1}) + t(v_{k+1})) + (d+1).
\end{align*}


\noindent Let $a = n^2 - k^2 + 2(d-1)(n-k) + (k-1)^2$, and suppose $i = 1$. Then $t(v_1) = t(v_2) + 1.$ If $k \neq 2$, then $t(v_1) = f(2) + 1 = a - (2-1)^2 + 1 = a = f(1)$. If $k = 2$, then $t(v_1) = n^2 - 2^2 + 2(d-1)(n-2) + 1 = f(1)$. 

\noindent Let $i = n-1$. Then $t(v_{n-1}) = \dfrac{1}{2}f(n-2) + 1$. If $k \neq n-1$, then 

\begin{align*}
    t(v_{n-1}) & = \dfrac{1}{2}f(n-2) + 1 \\
        & = \dfrac{1}{2}(n^2 - (n-2)^2 + 2(d-1)(n-(n-2)) + 1 \\
        & = 2n - 1 + 2(d-1) \\
        & = n^2 - (n-1)^2 + 2(d-1)(n-(n-1)) = f(n-1).
\end{align*}

\noindent Likewise if $i = n-1$ and $k = n-1$, then

\begin{align*}
    t(v_{n-1}) & = \dfrac{1}{2}f(n-2) + 1 \\
        & = \dfrac{1}{2}(n^2 - (n-2)^2 + 2(d-1)(n-(n-2)) + (n-2)^2 - (n-2)^2) + 1 \\
        & = 2n - 1 + 2(d-1) \\
        & = n^2 - (n-1)^2 + 2(d-1)(n-(n-1)) = f(n-1).
\end{align*}

\noindent Let $i \in \{2, \ldots, k-2\}$. Then $t(v_i) =  \dfrac{1}{2}(f(i-1) + f(i+1))  + 1$. Thus

\begin{align*}
    t(v_i) & =  \dfrac{1}{2}(f(i-1) + f(i+1))  + 1 \\
     & = \dfrac{1}{2}(a - (i-2)^2 + a - i^2) + 1 \\
     & = a - (i-1)^2 = f(i).
\end{align*}

\noindent Let $i \in \{k + 1, \ldots, n-2\}$. Then $t(v_i) =  \dfrac{1}{2}(f(i-1)+ f(i+1))  + 1$. Thus

\begin{align*}
    t(v_i) & =  \dfrac{1}{2}(f(i-1) + f(i+1))  + 1 \\
     & = \dfrac{1}{2}(n^2 - (i-1)^2 + 2(d-1)(n-(i-1)) +  n^2 - (i+1)^2 + 2(d-1)(n-(i+1))) + 1 \\
     & = n^2 - i^2 - 1+ 2(d-1)n - \dfrac{1}{2}(2(d-1)(i-1) + 2(d-1)(i+1)) + 1 \\
     & = n^2 - i^2 + 2(d-1)(n-i) = f(i).
\end{align*}

\noindent Let $i = k-1$. Then $t(v_{k-1}) = \dfrac{1}{2}(f(k-2) + f(k))  + 1$. Note that this ASUA equation requires the left and right pieces of $f(i)$. Thus,  

\begin{align*}
    t(v_{k-1}) & =  \dfrac{1}{2}(f(k-2) + f(k))  + 1 \\
     & = \dfrac{1}{2}(n^2 - k^2 + 2(d-1)(n-k) +  (k-1)^2 - (k-3)^2 +  n^2 - k^2 + 2(d-1)(n-k)) + 1 \\
     & = n^2 - k^2 + 2(d-1)(n-k) + 2k - 4 + 1 \\
     & = n^2 - k^2 + 2(d-1)(n-k) + k^2 - 2k + 1 - k^2+ 4k - 4\\
     & = n^2 - k^2 + 2(d-1)(n-k) + (k-1)^2 - (k-2)^2 = f(k-1).
\end{align*}

\noindent Lastly, let $i = k$. Then $t(v_{k}) = \dfrac{1}{2}(f(k-1) + f(k+1))  + (d+1)$. Thus  

\begin{align*}
    t(v_{k}) & =  \dfrac{1}{2}(f(k-1) + f(k+1))  + (d+1) \\
     & = \dfrac{1}{2}(n^2 - k^2 + 2(d-1)(n-k) +  (k-1)^2 - (k-2)^2 +  n^2 - k^2 + 2(d-1)(n-(k+1))) + (d+1) \\
     & = n^2 - k^2 + 2(d-1)(n-k) -(d-1) -2 + (d + 1) \\
     & = n^2 - k^2 + 2(d-1)(n-k)  = f(k).
\end{align*}

\noindent Thus $t(v_i) = f(i)$ for all $i$, $1 \leq i \leq n-1$.

\end{proof}

\section{Other Questions to Consider}

\noindent ASUA equations provide valuable insights in calculating average steps until absorption along a random walk. In particular, they focus on local structure and avoid the need to calculate inverses of arbitrarily large matrices. There are several variations to our setup that may be considered. We allowed steps in the walk to be equally weighted, though, by adding parallel edges, one could easily adjust probabilities to fit uneven weights on the edges. If one desired multiple absorption points on a graph, say $x$ and $y$, we note ASUAs for this graph are the same as a new graph in which an edge, $xy$, is added to the graph and then contracted.

\noindent For graph $G$, let $t_{\sigma}(G, u) = \sum_{V} t(v)$ be the sum total of ASUAs of vertices of $G$ with absorbing vertex $u$, and let $t'(G, v, u)$ be the average steps needed to traverse from $v$ to $u$ and back to $v$. Besides classifying ASUAs on other classes of graphs, we are interested in looking for sharp bounds for $t_{\sigma}$ and $t'$ for classes of graphs. For trees, we hypothesize that $t_{\sigma}(S_n) \leq T \leq t_{\sigma}(P_n)$ and $t'(S_n) \leq T \leq t'(P_n)$ for any tree $T$ of order $n$, where $S_n$ is the star on $n$ vertices.



    \nocite{*}
    \footnotesize{\bibliographystyle{plain}\bibliography{references}}
    
\end{document}